\documentclass[11pt]{amsart}
\usepackage[english]{babel}
\usepackage{graphicx}
\newtheorem{theorem}{Theorem}[section]

\newtheorem{lemma}[theorem]{Lemma}

\theoremstyle{definition}

\theoremstyle{remark}
\newtheorem{remark}{Remark}
\numberwithin{equation}{section}
\newcommand{\real}{\mathbb R}
\def\natu{\mathbb N}

\def\landan{\Lambda_n}
\def\landa0{\Lambda_0}
\def\landaene{\lambda_{2n-1}}
\def\periodica{L_{T}(\real,\real)}
\def\landaenea{\lambda_n (a)}

\begin{document}

\title[Lyapunov inequalities for the periodic bvp]
{Lyapunov inequalities for the periodic boundary value problem at higher eigenvalues}%
\author{Antonio Ca\~{n}ada}
\author{Salvador Villegas}
\thanks{The authors have been supported by the Ministry of Education and
Science of Spain (MTM2008.00988)}
\address{Departamento de An\'{a}lisis
Matem\'{a}tico, Universidad de Granada, 18071 Granada, Spain.}
\email{acanada@ugr.es, svillega@ugr.es}


%\date{}%
%\dedicatory{}%
%\commby{}%
% ----------------------------------------------------------------
\begin{abstract}
This paper is devoted to provide some new results on Lyapunov type
inequalities for the periodic boundary value problem at higher
eigenvalues. Our main result is derived from a detailed analysis
on the number and distribution of zeros of nontrivial solutions
and their first derivatives, together with the study of some
special minimization problems, where the Lagrange multiplier
Theorem plays a fundamental role. This allows to obtain the
optimal constants. Our applications include the Hill's equation
where we give some new conditions on its stability properties and
also the study of periodic and nonlinear problems at resonance
where we show some new conditions which allow to prove the
existence and uniqueness of solutions.
\end{abstract}

\maketitle

\section{Introduction}

\noindent Lyapunov type inequalities provide optimal necessary
conditions for certain linear homogeneous boundary value problems
to have nontrivial solutions. For instance, if the function $a$
satisfies
\begin{equation}\label{condiciones}
a \in \periodica \setminus \{ 0 \}, \ \int_0^T  a(x) \ dx \geq 0
\end{equation}
where $L_{T}(\real,\real)$ denotes the set of $T-$periodic
functions $a: \mathbb{R} \rightarrow \mathbb{R},$ such that
$a|_{[0,T]} \in L^1 (0,T),$ then it may be proved (see
\cite{huaizhongyongjde94}) that if the periodic boundary value
problem
\begin{equation}\label{p1}
u''(x) + a(x)u(x) = 0, \ x \in (0,T), \ u(0)-u(T) = u'(0)- u'(T) =
0
\end{equation}
has nontrivial solutions, then
\begin{equation}\label{p2bis}
\displaystyle \int_{0}^{T} a^+ (x) \ dx
> 16/T
\end{equation}
where $a^{+}(x) = \max \{ a(x),0 \}.$ This fact has a trivial
consequence: if $a$ satisfies (\ref{condiciones}) and
$\displaystyle \int_0^T a^+ (x) \ dx  \leq 16/T,$ then the unique
solution of (\ref{p1}) is the trivial one.  Moreover, this
constant is optimal: for any constant $k > 16/T,$ there is some
function $a$ satisfying (\ref{condiciones}) and $ \displaystyle
\int_0^T a^+ (x) \ dx  \leq k,$ such that (\ref{p1}) has
nontrivial solutions (\cite{huaizhongyongjde94}).

This kind of result has been generalized in different ways. For
example, in \cite{camovimia}, \cite{cavidcds} and \cite{zhang},
the authors provide, for each $p$ with $1 \leq p \leq \infty,$
optimal necessary conditions for boundary value problems similar
to (\ref{p1}) to have nontrivial solutions, given in terms of the
$L^p$ norm of the function $a^+.$ This includes the case of
Dirichlet, Neumann or mixed boundary conditions. Also, some
Lyapunov inequalities may be obtained for $q-$Laplacian operators
(\cite{zhang}) and for elliptic PDE (\cite{camovijfa}).

Let us observe that the real number zero plays a fundamental role
in the condition (\ref{condiciones}). The number zero is,
precisely, the first (or principal) eigenvalue of the eigenvalue
problem
\begin{equation}\label{p2}
 u''(x) + \lambda u(x) = 0, \ x \in (0,T), \ u(0)-u(T) = u'(0)- u'(T) =
0 \end{equation} and if $a \in \periodica,$ a sufficient condition
to get (\ref{condiciones}) is the condition $0 \prec a$ where
where for $c,d \in L^1 (0,T),$ we write $c \prec d$ if $c(x) \leq
d(x)$ for a.e. $x \in [0,T]$ and $c(x) < d(x)$ on a set of
positive measure.

On the other hand, the set of eigenvalues of (\ref{p2}) is given
by $\lambda_0 = 0, \ \lambda_{2n-1} = \lambda_{2n}= (2n)^2
\pi^2/T^2, \ n \in \natu$ and it is clear that if for some
sufficiently large $n \in \natu,$ the function $a$ satisfies
$\lambda_{2n-1} \prec a,$ then the inequality $\int_0 ^T a^+ (x) \
dx \leq \frac{16}{T}$ is not possible. To this respect, the first
part of this paper deals with $L_1-$Lyapunov inequality for the
periodic problem (\ref{p1}) at higher eigenvalues. More precisely,
if $n \in \natu$ is fixed, we introduce the set $\landan$ as
\begin{equation}
\landan = \{ a \in \periodica: \lambda_{2n-1} \prec a \ \mbox{and}
\ (\ref{p1}) \ \mbox{has nontrivial solutions} \ \}
\end{equation}
and we give an explicit expression for the number
\begin{equation}
\gamma_{1,n} = \inf_{a \in \landan } \ \Vert a  \Vert_{L^1 (0,T)}
\end{equation} In addition, we prove that this infimum is not attained. To the best of our
knowledge these results are new if $n \geq 1$. In Remarks
\ref{010609t1}, \ref{1212073} and \ref{1604093} below we compare
our results with others obtained by different authors.

Our main result is derived from a detailed analysis on the number
and distribution of zeros of nontrivial solutions and their first
derivatives, together with the study of some special minimization
problems. We apply our method to the periodic and also to the
antiperiodic boundary value problem.

In the last section of the paper we give some applications. In
particular, we obtain in a very easy way a generalization of a
result previously proved by Krein (\cite{krein}), on the so called
$n^{th}$ stability zone of the Hill's equation. Also, we compare
our results with those obtained by Borg in \cite{borg}, about the
absolute stability of the Hill's equation with two parameters.
Moreover, we use an especial continuation method to prove the
positivity of some eigenvalues; this provides some new stability
results (see Subsection \ref{subsection1} below).

Finally, we present some new conditions which allow to prove the
existence and uniqueness of solutions for some nonlinear periodic
and resonant problems.

\section{The periodic problem}

If $n \in \natu $ is fixed, we introduce the set $\landan$ as

\begin{equation}\label{p2806073}
\landan = \{ a \in \periodica : \lambda_{2n-1} \prec a \
\mbox{and} \ (\ref{p1}) \ \mbox{has nontrivial solutions} \ \}
\end{equation}
(Remember that $\lambda_{2n-1}= \lambda_{2n} = 4n^2\pi^2/T^2$). If
$a \in \landan,$ and $u$ is any nontrivial solution of (\ref{p1}),
then $u$ is not a constant function. In addition, $u$ must have a
zero in the interval $[0,T].$  If $r \in [0,T]$ is such that $u(r)
=0,$ the periodic and nontrivial function $v(x) = u(r+x)$
satisfies $v''(x) + a(r +x)v(x) = 0, \ x \in (0,T)$ and $\Vert a(r
+ \cdot) - \landaene \Vert_{L^1 (0,T)}= \Vert a(\cdot) - \landaene
\Vert_{L^1 (0,T)}.$ Finally, since $a \in \landan,\ n \in \natu,$
it is clear that between two consecutive zeros of the function $u$
there must exists a zero of the function $u'$ and between two
consecutive zeros of the function $u'$ there must exists a zero of
the function $u.$

Previously to state and prove the main result in this section, we
remember that for any function $a \in \periodica,$ the eigenvalues
for
\begin{equation}\label{p2bisbis}
 u''(x) + (\lambda+a(x)) u(x) = 0, \ x \in (0,T), \ u(0)-u(T) = u'(0)- u'(T) =
0 \end{equation} form a sequence $\landaenea, \ n \in \natu \cup
\{ 0 \},$ such that
\begin{equation}\label{1604091}
\lambda_0 (a) < \lambda_1 (a) \leq \lambda_2 (a)< \ldots <
\lambda_{2n-1}(a) \leq \lambda_{2n}(a) < \ldots
\end{equation}
with $\lambda_0(a)$ simple and such that if $\phi_n$ is the
corresponding eigenfunction to $\lambda_n(a),$ then $\phi_0$ has
no zeros in $[0,T]$ and $\phi_{2n-1}$ and $\phi_{2n}$ have exactly
$2n$ zeros in $[0,T)$ (see \cite{cole}). In particular, $\lambda_0
= \lambda_0 (0) = 0, \ \lambda_{2n-1} = \lambda_{2n}=
\lambda_{2n-1}(0) = \lambda_{2n}(0)= (2n)^2 \pi^2/T^2, \ n \in
\natu.$

\begin{theorem}\label{pt1}
Let $n \in \natu$ and $a \in \landan$ be given and $u$ any
nontrivial solution of (\ref{p1}) such that $u(0) = u(T) = 0.$ If
the zeros of $u$ in $[0,T]$ are denoted by $0 = x_0 < x_2 < \ldots
< x_{2m} = T$ and the zeros of $u'$ in $(0,T)$ are denoted by $
x_1 < x_3 < \ldots < x_{2m-1}, $ then:
\begin{enumerate}
\item $x_{i+1} - x_{i} \leq \frac{T}{4n}, \ \forall \ i: \ 0 \leq
i \leq 2m-1.$ Moreover, at least one of these inequalities is
strict. \item $m$ is an even number and $m \geq 2(n+1).$ Any even
value $m\geq 2(n+1)$ is possible. %
\item
\begin{equation}\label{02071}
\Vert a - \landaene \Vert_{L^1 (x_i,x_{i+1})} \geq \frac{2n\pi}{T}
\cot (\frac{2n\pi}{T}(x_{i+1}-x_i)),\ 0 \leq i \leq 2m-1.
\end{equation}
\item
\begin{equation}\label{p1009071}
\beta_{1,n} \equiv \inf_{a \in \landan } \ \Vert a -
\lambda_{2n-1} \Vert_{L^1 (0,T)} = \frac{8\pi n (n+1)}{T} \cot
\frac{n\pi}{2(n+1)}
\end{equation}
and $\beta_{1,n}$ is not attained.%
\item If $a \in \periodica$ satisfies
\begin{equation}\label{positivo}
\lambda_{2n-1} \prec a, \ \Vert a \Vert_{L^1(0,T)} \leq
\gamma_{1,n} \end{equation} where \begin{equation}\label{0909092}
\gamma_{1,n} = T\lambda_{2n-1} + \beta_{1,n},
\end{equation}
then
\begin{equation}\label{1604096}
\lambda_{2n}(a) < 0 < \lambda_{2n+1}(a)
\end{equation}
\end{enumerate}
\end{theorem}

\begin{proof}

Let $i, \ 0 \leq i \leq 2m-1,$ be given. Then, function $u$
satisfies either the problem \begin{equation}\label{2806076}
u''(x) + a(x)u(x) = 0, \ x \in (x_i,x_{i+1}), \ \ u(x_i) = 0, \
u'(x_{i+1}) = 0, \end{equation} or the problem
\begin{equation}\label{2806077} u''(x) + a(x)u(x) = 0, \ x \in (x_i,x_{i+1}),
\ \ u'(x_i) = 0, \ u(x_{i+1}) = 0. \end{equation} Let us assume
the first case. The reasoning in the second case is similar. Note
that $u$ may be chosen such that $u(x)
> 0, \ \forall \ x  \in  (x_i,x_{i+1}).$ Let us denote by $\mu_1 ^i$ and
$\varphi_1^i,$ respectively, the principal eigenvalue and
eigenfunction of the eigenvalue problem
\begin{equation}\label{2806074} v''(x) + \mu v(x) = 0, \ x \in
(x_i,x_{i+1}), \ v(x_i) = 0, \ v'(x_{i+1}) = 0. \end{equation} It
is known that
\begin{equation}\label{2806075}
\mu_1 ^i = \frac{\pi^2}{4(x_{i+1}-x_i)^2}, \ \varphi_1 ^i (x) =
\sin \frac{\pi (x-x_i)}{2(x_{i+1}-x_i)} \end{equation} Choosing
$\varphi_1 ^i$ as test function in the weak formulation of
(\ref{2806076}) and $u$ as test function in the weak formulation
of (\ref{2806074}) for $\mu = \mu_1 ^i$ and $v =\varphi_1^i,$ we
obtain
\begin{equation}\label{2806078}
\int_{x_i}^{x_{i+1}} (a(x) - \mu_1^i)u\varphi_1^i (x) \ dx = 0.
\end{equation}
Then, if $x_{i+1} - x_{i} > \frac{T}{4n},$ we have
$$
\mu_1 ^i = \frac{\pi^2 T^2}{4(x_{i+1}-x_i)^2T^2} < \frac{4n^2
\pi^2}{T^2} = \landaene \leq a(x), \ \mbox{a.e. in} \
(x_i,x_{i+1})
$$
which is a contradiction with (\ref{2806078}). Consequently,
$x_{i+1} - x_i \leq \frac{T}{4n}, \forall \ i: \ 0 \leq i \leq
2m-1.$ Also, since $\landaene \prec a$ in the interval $(0,T)$, we
must have $\landaene \prec a$ in some subinterval $(x_j,x_{j+1})$.
If $x_{j+1} - x_{j}= \frac{T}{4n},$ it follows $\mu_1 ^j \prec a$
in $(x_j,x_{j+1})$ and this is again a contradiction with
(\ref{2806078}). These reasonings complete the first part of the
theorem.

For the second one, let us observe that
\begin{equation}\label{1307091}
T = \sum_{i=0}^{2m-1} (x_{i+1}-x_i) < 2m \frac{T}{4n}
\end{equation}
In consequence, $m >2n$ and therefore $m \geq 2n+1.$ At this
point, we claim that, for the periodic problem (\ref{p1}), $m$
must be an even number. This implies $m \geq 2(n+1).$

To prove the claim, let us observe that $u(x_0) = 0$ and $u'(x_0)
\neq 0.$ Assume, for instance that $u'(x_0) > 0.$ Then, we have
$$
\begin{array}{c}
u(x_0) = 0, \ u'(x_0) > 0; \ u(x_1) > 0, \ u'(x_1) = 0, \\
u(x_2) = 0, \ u'(x_2) < 0; \ u(x_3) < 0, \ u'(x_3) = 0, \\
u(x_4) = 0, \ u'(x_4) > 0; \ u(x_5) > 0, \ u'(x_5) = 0, \\
\ldots
\end{array}
$$
Since $u'(x_0)u'(x_{2m})= u'(0)u'(T) > 0,$ $m$ must be an even
number. Also, note that for any given even and natural number $q
\geq 2(n+1),$ function $b(x) \equiv \lambda_{q}$ belongs to
$\landan$ and for function $v(x) = \sin \frac{q\pi x}{T},$ we have
$m = q.$ In this way, we have proved the first two parts of the
Theorem.

$ $

Continuing with the proof of the Theorem, if $i,$ with $ \ 0 \leq
i \leq 2m-1$ is given and $u$ satisfies (\ref{2806076}), then
$$
\begin{array}{c}
\int_{x_i}^{x_{i+1}} u'^2(x) = \int_{x_i}^{x_{i+1}} a(x)u^2(x) =
\\ \\
\int_{x_i}^{x_{i+1}} (a(x)-\landaene) u^2 (x) +
\int_{x_i}^{x_{i+1}} \landaene u^2 (x)
\end{array}
$$
Therefore, $$ \int_{x_i}^{x_{i+1}} u'^2(x) - \landaene
\int_{x_i}^{x_{i+1}} u^2(x) \leq \Vert a - \landaene \Vert_{L^1
(x_i,x_{i+1})} \Vert u^2 \Vert_{L^\infty (x_i,x_{i+1})}
$$
Since $u'$ has no zeros in the interval $(x_i,x_{i+1})$ and
$u(x_i) = 0,$ we have $\Vert u^2 \Vert_{L^\infty(x_i,x_{i+1})} =
u^2 (x_{i+1}).$ This proves
\begin{equation}\label{202071}
\Vert a - \landaene \Vert_{L^1 (x_i,x_{i+1})} \geq
\frac{\int_{x_i}^{x_{i+1}} u'^2 - \landaene \int_{x_i}^{x_{i+1}}
u^2 }{u^2(x_{i+1})}.
\end{equation}
At this point, the following Lemma (\cite{cavijems}, Lemma 2.3.)
may be useful.

\begin{lemma}\label{l2}
Assume that $a < b$ and $0 < M \leq \frac{\pi^2}{4(b-a)^2}$ are
given real numbers. Let $H = \{ u \in H^1 (a,b) : u(a) = 0, u(b)
\neq 0 \}.$ If $J: H \rightarrow \real$ is defined by
\begin{equation}\label{02073}
J(u) = \frac{\int_{a}^{b} u'^2 - M \int_{a}^{b} u^2 }{u^2(b)}
\end{equation}
and $c \equiv \inf_{u \in H} \ J(u),$ then $c$ is attained.
Moreover
\begin{equation}
c = M^{1/2}\cot (M^{1/2}(b-a))
\end{equation}
and if $u \in H,$ then $J(u) = c \Longleftrightarrow u(x) = k
\frac{\sin (M^{1/2}(x-a))}{\sin (M^{1/2}(b-a))}$ for some non zero
constant $k.$
\end{lemma}

By using this Lemma in (\ref{202071}) with $a= x_i, b= x_{i+1}, M=
\lambda_{2n-1}, $ we deduce the third part of the Theorem. In
particular, this implies
\begin{equation}\label{2705091}
\Vert a - \landaene \Vert_{L^1 (0,T)} \geq \frac{2n\pi}{T}
\sum_{i=0}^{2m-1} \cot (\frac{2n\pi}{T} (x_{i+1} -x_i))
\end{equation}
The right-hand side of (\ref{2705091}) attains its minimum if and
only if $x_{i+1}-x_i = \frac{T}{2m}, \ 0 \leq i \leq 2m-1$ (see
Lemma 2.5 in \cite{cavijems}) so that
\begin{equation}\label{2805091} \beta_{1,n} \geq \frac{4n\pi}{T}m
\cot \frac{n\pi}{m}
\end{equation}
Taking into account that the function $ m \cot \frac{n\pi}{m}$ is
strictly increasing with respect to $ m,$ and that $m \geq
2(n+1),$ we deduce
\begin{equation}\label{03071}
\beta_{1,n}  \ \geq \frac{8\pi n(n+1)}{T} \cot
\frac{n\pi}{2(n+1)}.
\end{equation}

In the next Lemma, we define a minimizing sequence for
$\beta_{1,n}.$ To this respect, we construct, for each positive
and sufficiently small number $\varepsilon,$ an appropriate
$T-$periodic function $u_{\varepsilon}\in C^2 [0,T]$ in the
following way:
\begin{enumerate}
\item $u_{\varepsilon} $ is first explicitly defined in
$[0,\varepsilon]$ as
$$
u_\varepsilon (x) = -\sin (\frac{2n\pi}{T}(x-\frac{T}{4(n+1)})) +
h(x,\varepsilon)$$ where $h(x,\varepsilon)$ is such that
$a_{\varepsilon} (x) = \frac{-u_{\varepsilon} ''
(x)}{u_{\varepsilon}(x)} > \lambda_{2n-1}, \ \forall \ x \in
[0,\varepsilon].$ \item $u_\varepsilon (x) = -\sin
(\frac{2n\pi}{T}(x-\frac{T}{4(n+1)})), \ \forall \ x \in
(\varepsilon,\frac{T}{4(n+1)}].$ Then, $a_{\varepsilon} (x) =
\frac{-u_{\varepsilon} '' (x)}{u_{\varepsilon}(x)} \equiv
\lambda_{2n-1}, \ \mbox{in the interval}\
(\varepsilon,\frac{T}{4(n+1)}].$ \item
$$ \liminf_{\varepsilon \rightarrow
0^+} \ \Vert a_\varepsilon - \landaene \Vert_{L^1 (0,
\frac{T}{4(n+1)})} = \frac{2n\pi}{T} \cot \frac{n\pi}{2(n+1)}
$$
\item In the interval $[0,\frac{2T}{4(n+1)}],$ the function
$u_{\varepsilon}$ is an odd function with respect to
$\frac{T}{4(n+1)}.$ \item In the interval $[0,\frac{4T}{4(n+1)}],$
the function $u_{\varepsilon}$ is an even function with respect to
$\frac{2T}{4(n+1)}.$ \item The function $u_{\varepsilon}$ is a
periodic function with period $\frac{T}{(n+1)}.$
\end{enumerate}

The details are given in the next Lemma.

\begin{lemma}\label{l7}
Let $ \varepsilon
> 0$ be sufficiently small. Let us define the function
$u_\varepsilon : [0,T] \rightarrow \real$ by
\begin{equation}
u_\varepsilon (x) = \left \{
\begin{array}{l}
-\sin (\frac{2n\pi}{T}(x-\frac{T}{4(n+1)})) +
\frac{2n\pi}{T}\frac{(x-\varepsilon)^3}{3\varepsilon^2}\cos
(\frac{n\pi}{2(n+1)}), \ \ \mbox{if}  \ \ 0 \leq x \leq \varepsilon, \\
\\
-\sin (\frac{2n\pi}{T}(x-\frac{T}{4(n+1)})), \ \ \mbox{if} \ \
\varepsilon \leq x \leq \frac{T}{4(n+1)}, \\ \\
-u_\varepsilon (\frac{2T}{4(n+1)} -x), \ \ \mbox{if} \ \
\frac{T}{4(n+1)} \leq x \leq \frac{2T}{4(n+1)}, \\ \\
u_\varepsilon (\frac{4T}{4(n+1)}-x), \ \ \mbox{if} \ \
\frac{2T}{4(n+1)} \leq x \leq \frac{4T}{4(n+1)}, \\ \\
%-u_\varepsilon (\frac{6T}{4(n+1)}-x), \ \ \mbox{if} \ \
%\frac{4T}{4(n+1)} \leq x \leq \frac{6T}{4(n+1)},  \\ \\
%\ldots
u_{\varepsilon} \ \mbox{is extended to the interval} \ [0,T] \
\mbox{as a} \ \frac{T}{n+1}-\mbox{periodic function.}
\end{array}
\right.
\end{equation}
Then $u_\varepsilon \in C^2[0,T]$, the function $a_\varepsilon(x)
\equiv \frac{-u_\varepsilon ''(x)}{u_\varepsilon (x)}, \ \forall \
x \in [0,T], \ x \neq \frac{(2k-1)T}{4(n+1)},  \ 1 \leq k \leq
2(n+1),$ belongs to $\landan$ and
\begin{equation}\label{0307t2}\liminf_{\varepsilon \rightarrow
0^+} \ \Vert a_\varepsilon - \landaene \Vert_{L^1 (0,T)} =
\frac{8\pi n(n+1)}{T} \cot\frac{n\pi}{2(n+1)}
\end{equation}
\end{lemma}
\begin{proof}
We claim that for each $0 \leq i \leq 4n+3,$ function
$a_\varepsilon$ satisfies
\begin{equation}\label{0307t1}
\landaene \prec a_\varepsilon, \ \mbox{in the interval } \ \left (
\frac{iT}{4(n+1)}, \frac{(i+1)T}{4(n+1)} \right )
\end{equation}
and
\begin{equation}\label{0307t3}
\liminf_{\varepsilon \rightarrow 0^+} \ \Vert a_\varepsilon -
\landaene \Vert_{L^1 (\frac{iT}{4(n+1)}, \frac{(i+1)T}{4(n+1)})} =
\frac{2n\pi}{T} \cot \frac{n\pi}{2(n+1)}
\end{equation}
It is trivial that from (\ref{0307t1}) and (\ref{0307t3}) we
deduce (\ref{0307t2}). Moreover, taking into account the
definition of the function $u_\varepsilon,$ it is clear that it is
sufficient to prove the claim in the case $i =0.$ Now, if $x \in
(0,\frac{T}{4(n+1)})$ we can distinguish two cases:
\begin{enumerate}
\item $x \in (\varepsilon, \frac{T}{4(n+1)}).$ Then $a_\varepsilon
(x) = \frac{-u_\varepsilon ''(x)}{u_\varepsilon (x)} \equiv
\landaene.$ \item $x \in (0,\varepsilon).$ Then
\end{enumerate}
$$
a_\varepsilon (x) - \landaene = \frac{-4
\frac{x-\varepsilon}{\varepsilon^2} \frac{n\pi}{T} \cos
\frac{n\pi}{2(n+1)} -  8\frac{(x-\varepsilon)^3}{3\varepsilon^2}
\frac{n^3\pi^3}{T^3} \cos \frac{n\pi}{2(n+1)}}{-\sin
(\frac{2n\pi}{T}(x-\frac{T}{4(n+1)})) +2
\frac{(x-\varepsilon)^3}{3\varepsilon^2} \frac{n\pi}{T} \cos
\frac{n\pi}{2(n+1)}} > 0
$$
Therefore $a_\varepsilon \in \Lambda_n.$ Moreover, if $\varepsilon
\rightarrow 0^+,$ then
$$
\frac{- 8\frac{(x-\varepsilon)^3}{3\varepsilon^2}
\frac{n^3\pi^3}{T^3} \cos \frac{n\pi}{2(n+1)}}{-\sin
(\frac{2n\pi}{T}(x-\frac{T}{4(n+1)})) +2
\frac{(x-\varepsilon)^3}{3\varepsilon^2} \frac{n\pi}{T} \cos
\frac{n\pi}{2(n+1)}} \rightarrow 0,
$$
uniformly if $x \in (0,\varepsilon).$

Finally, since
$$
\lim_{\varepsilon \rightarrow 0^+} \int_0^{\varepsilon} \left [
\frac{-4 \frac{x-\varepsilon}{\varepsilon^2} \frac{n\pi}{T} \cos
\frac{n\pi}{2(n+1)}} {-\sin (\frac{2n\pi}{T}(x-\frac{T}{4(n+1)})) +2
\frac{(x-\varepsilon)^3}{3\varepsilon^2} \frac{n\pi}{T} \cos
\frac{n\pi}{2(n+1)}} - \frac{-4
\frac{x-\varepsilon}{\varepsilon^2} \frac{n\pi}{T} \cos
\frac{n\pi}{2(n+1)}} {-\sin (\frac{2n\pi}{T}(x-\frac{T}{4(n+1)})) }
\right ] = 0
$$
and
$$-\sin (\frac{2n\pi}{T}(x-\frac{T}{4(n+1)})) \rightarrow \sin
\frac{n\pi}{2(n+1)},$$ uniformly  in $x \in (0,\varepsilon)$ when
$\varepsilon \rightarrow 0+,$ we deduce
$$
\liminf_{\varepsilon \rightarrow 0^+} \ \Vert a_{\varepsilon} -
\landaene \Vert_{L^1 (0,\frac{T}{4(n+1)})} = \liminf_{\varepsilon
\rightarrow 0^+} \frac{n\pi}{T} \cot \frac{n\pi}{2(n+1)}
\frac{4}{\varepsilon^2}\int_0^{\varepsilon} (\varepsilon - x) =
\frac{2n\pi}{T} \cot \frac{n\pi}{2(n+1)}
$$
which is (\ref{0307t3}) for the case $i=0$.
\end{proof}

In the next Lemma we prove that the infimum $\beta_{1,n}$ is not
attained. The key point in the proof is the optimality property of
the different inequalities which have been obtained previously.

\begin{lemma}\label{l6}
$\beta_{1,n}$ is not attained.
\end{lemma}

\begin{proof}
Let $a \in \landan$ be such that $\Vert a - \landaene \Vert_{L^1
(0,T)} = \beta_{1,n}.$ Let $u$ be any nontrivial solution of
(\ref{p1}) associated to the function $a$. As previously, we
denote the zeros of $u$ by $0 = x_0 < x_2 < \ldots < x_{2m} = T$
and the zeros of $u'$ by $ x_1 < x_3 < \ldots < x_{2m-1}. $ By
using (\ref{02071}) and Lemma \ref{l2}, we have
\begin{equation}\label{1209071}
\begin{array}{c}
\beta_{1,n} = \Vert a - \landaene \Vert_{L^1 (0,T)} =
\displaystyle \sum_{i=0}^{2m-1}  \Vert a - \landaene \Vert_{L^1
(x_i,x_{i+1})} \geq \\ \\
\sum_{i=0}^{2m-1} J_i (u) \geq \frac{2n\pi}{T}\displaystyle
\sum_{i=0}^{2m-1} \cot \frac{2n\pi (x_{i+1}-x_i)}{T} \geq \\ \\
\frac{4\pi n m}{T}  \cot \frac{n\pi}{m} \geq \frac{8\pi n
(n+1)}{T} \cot \frac{n\pi}{2(n+1)} = \beta_{1,n}
\end{array}
\end{equation}
where $J_i (u)$ is given either by $$ J_i(u) =
\frac{\int_{x_i}^{x_{i+1}} u'^2 - \landaene \int_{x_i}^{x_{i+1}}
u^2 }{u^2(x_{i+1})}, \ \ \mbox{if} \ u(x_i) = 0 $$ or by
$$
J_i(u)= \frac{\int_{x_i}^{x_{i+1}} u'^2 - \landaene
\int_{x_i}^{x_{i+1}} u^2 }{u^2(x_i)}, \ \ \mbox{if} \ u(x_{i+1}) =
0. $$ Consequently, all inequalities in (\ref{1209071}) transform
into equalities. In particular we obtain that
$$
m = 2(n+1), \ x_{i+1}-x_i = \frac{T}{4(n+1)}, \ 0 \leq i \leq
4n+3.
$$
Also, it follows
$$ J_i (u) = \frac{2n\pi}{T} \cot \frac{2n\pi}{ T}\frac{T}{ 4 (n+1)}, \
0 \leq i \leq 4n+3.
$$
From Lemma \ref{l2} we deduce that, up to some nonzero constants,
function $u$ fulfils in each interval $[x_i,x_{i+1}],$
$$
u(x) = \frac{\sin \frac{2n\pi}{T} (x-x_i)}{\sin \frac{2n\pi}{T}
(x_{i+1}-x_i)}, \ \mbox{if} \ i \ \mbox{is even},
$$
$$
\mbox{and}\ \ u(x) = \frac{\sin \frac{2n\pi}{T} (x-x_{i+1})}{\sin
\frac{2n\pi}{T} (x_{i}-x_{i+1})}, \ \mbox{if} \ \ i \ \mbox{is
odd}.
$$
In particular, in the interval $[0,\frac{T}{4(n+1)}]= [x_0,x_1],$
$u$ must be the function
$$
u(x) = \frac{\sin \frac{2n\pi}{T}(x)}{\sin
\frac{2n\pi}{T}(\frac{T}{4(n+1)})}
$$
which does not satisfy the condition $u'(x_1) = 0.$ The conclusion
is that $\beta_{1,n}$ is not attained.
\end{proof}

For the last part of the theorem, let us assume that the function
$a$ satisfies (\ref{positivo}). Then, since $\lambda_{2n-1} \prec
a,$ we trivially have $\lambda_{2n}(a) <
\lambda_{2n}(\lambda_{2n-1}) = 0.$ To prove that $\lambda_{2n+1}
(a) >0$ we use a continuation method: let us define the continuous
function $g: [0,1] \rightarrow \real$ by
$$
g(\varepsilon) = \lambda_{2n+1}(a_{\varepsilon}(\cdot))
$$
where $a_{\varepsilon}(x) = \lambda_{2n-1} + \varepsilon (a(\cdot)
- \lambda_{2n-1}).$ Then $g(0) =\lambda_{2n+1}\left (
\lambda_{2n-1} \right )= \lambda_{2n+1} - \lambda_{2n-1} > 0.$
Moreover, $g(\varepsilon) \neq 0, \ \forall \ \varepsilon \in
(0,1].$ In fact, for each $\varepsilon \in (0,1]$ the function
$a_{\varepsilon}(x)$ satisfies $\lambda_{2n-1} \prec
a_{\varepsilon}$ and $\Vert a_{\varepsilon}(\cdot) -
\lambda_{2n-1} \Vert_{L^1 (0,T)} \leq \beta_{1,n}.$ Consequently,
we deduce from the previous parts of the Theorem that the number
$0$ is not an eigenvalue of the function $a_{\varepsilon}$ for the
periodic boundary conditions. As a consequence,
$\lambda_{2n+1}(a)= g(1)
> 0$ and the Theorem is proved.

\end{proof}

\begin{remark}\label{010609t1} Let us observe that we can obtain similar results if, in the definition
of the set $\Lambda_n$ in (\ref{p2806073}), we consider $n \in
\real^+$ instead of $n \in \natu.$ Only some minor changes are
necessary. From this point of view, if we consider $\beta_{1,n}$
as a function of $n \in (0,+\infty),$ then $\lim_{n \rightarrow
0^+} \ \beta_{1,n} = \frac{16}{T},$ the constant of the classical
$L^1$ Lyapunov inequality at the first eigenvalue which was
obtained in (\cite{huaizhongyongjde94}) by using methods of
optimal control theory. In fact, we can use similar reasonings to
those of Theorem \ref{pt1} if $a \in \Lambda_0$ where
\begin{equation}
{\small \Lambda_0 = \{ a \in \periodica \setminus \{ 0 \}: 0 \leq
\int_0^T a(x) \ dx \ \ \mbox{and} \ (\ref{p1}) \ \mbox{has
nontrivial solutions} \ \} }
\end{equation}
In this case $m\geq 2$ and any even value $m \geq 2$ is possible.
Consequently
\begin{equation}\label{1307094}
\beta_{1,0} = \inf_{a \in \Lambda_0} \ \Vert a^+ \Vert_{L^1 (0,T)}
= \displaystyle \frac{16}{T}
\end{equation}
Let us remark that the restriction
$$
a \in \periodica \setminus \{ 0 \}: \lambda_0 = 0 \leq \int_0^T
a(x) \ dx
$$
is more general that the restriction $ \lambda_0  \prec a.$
\end{remark}

\begin{remark}\label{1212073}
 The case where $T = 2\pi$ and
function $a$ satisfies the condition $ A \leq a(x) \leq B, \
\mbox{a.e. in} \ (0,2\pi) $ where $k^2 < A < (k+1)^2 < B$ for some
$k \in \natu \cup \{ 0 \}, $ has been considered in
\cite{wanglina95}, where the authors also use optimal control
theory methods. In this paper, the authors define the set
$\Lambda_{A,B}$ as the set of functions $a $ such that $A \leq
a(x) \leq B, \ \mbox{a.e. in} \ (0,T)$ and (\ref{p1}) has
nontrivial solutions. Then, by using the Pontryagin's maximum
principle they prove that the number
$$ \beta_{A,B} \equiv \inf_{a \in
\Lambda_{A,B} } \ \Vert a \Vert_{L^1 (0,T)} $$ is attained. In
addition, they calculate $\lim_{B \rightarrow +\infty} \
\beta_{A,B}.$ However, if $A \rightarrow k^2,$ it does not seem
possible to deduce from \cite{wanglina95} that the constant
$\beta_{1,k}$ (defined in (\ref{p1009071})) is not attained. In
fact, to the best of our knowledge, this result is new. Moreover,
our method, which combines a detailed analysis about the number
and distribution of zeros of nontrivial solutions of (\ref{p1})
and their first derivatives, together with the use of some special
minimization problems, can be used to unify other results obtained
for different boundary conditions (see \cite{cavijems} for the
Neumann and Dirichlet problem).
\end{remark}

\begin{remark}\label{1604093}
In \cite{zhang}, the author proves that if the function $a\in
\periodica$ satisfies
$$
\lambda_{2n-1}(0) \prec a, \ \Vert a \Vert_{L^1 (0,T)} \leq
\frac{16(n+1)^2}{T},
$$
then $\lambda_{2n}(a) < 0 < \lambda_{2n+1}(a).$ It it trivial that
$$
\frac{16(n+1)^2}{T} < \gamma_{1,n}, \ \forall \ n \in \mathbb{N}
$$
and that
$$
\displaystyle \lim_{n \rightarrow + \infty} \displaystyle \frac{T
\gamma_{1,n}}{16(n+1)^2} = \displaystyle \frac{\pi^2}{4}
$$
Therefore, the results given in Theorem \ref{pt1} is more precise.
Moreover, taking into account the definition of $\gamma_{1,n}$
(see (\ref{positivo})), if $\lambda_{2n-1} \prec a,$ our result is
optimal from the point of view of the nonexistence of nontrivial
solution of (\ref{p1}). On the other hand, in \cite{zhang} the
author studies the case of Dirichlet, periodic and antiperiodic
boundary conditions for one dimensional $q-$Laplacian operators,
by using $L^p-$ norms of the function $a$ (see also
\cite{camovimia} for Lyapunov inequalities at the first
eigenvalue).
\end{remark}

\section{The antiperiodic problem}

We can do an analogous study for other boundary value problems. In
fact, the key point of the method used in Theorem \ref{pt1} is to
have an optimal knowledge about the number and distribution of
zeros of the function $u$ and its derivative $u',$ moreover of
knowing the best value of the constant $m.$ Thinking in the next
section, we consider the anti-periodic boundary value problem
\begin{equation}\label{ap1}
u''(x) + a(x)u(x) = 0, \ x \in (0,T), \ u(0)+u(T) = u'(0)+ u'(T) =
0
\end{equation}
where $a \in \periodica.$ To this respect, it is very well known
that for any function $a \in \periodica,$ the eigenvalues for
\begin{equation}\label{antip2bisbis}
 u''(x) + (\tilde{\lambda}+a(x)) u(x) = 0, \ x \in (0,T), \ u(0)+u(T) = u'(0)+ u'(T) =
0 \end{equation} form a sequence $\tilde{\lambda}_{n} (a), \ n \in
\natu,$ such that
\begin{equation}\label{anti1604091}
\tilde{\lambda}_1 (a) \leq \tilde{\lambda}_2 (a)< \ldots <
\tilde{\lambda}_{2n-1}(a) \leq \tilde{\lambda}_{2n}(a) < \ldots
\end{equation}
such that if $\tilde{\phi}_n$ is the corresponding eigenfunction
to $\tilde{\lambda}_n(a),$ then $\tilde{\phi}_{2n-1}$ and
$\tilde{\phi}_{2n}$ have exactly $2n-1$ zeros in $[0,T)$ (see
\cite{cole}). In particular, the set of eigenvalues of
\begin{equation}\label{antip2}
 u''(x) + \lambda u(x) = 0, \ x \in (0,T), \ u(0)+u(T) = u'(0)+ u'(T) =
0 \end{equation} is given by $\tilde{\lambda}_{2n-1}(0) =
\tilde{\lambda}_{2n}(0)= (2n-1)^2 \pi^2/T^2, \ n \in \natu.$ We
will denote $\tilde{\lambda}_i = \tilde{\lambda}_i (0), \ \forall
\ i \in \natu.$

$ $

If $n \in \natu$ is fixed, we can introduce the set
$\tilde{\Lambda}_{n}$ as
\begin{equation}
\tilde{\Lambda}_{n} = \{ a \in \periodica: \tilde{\lambda}_{2n-1}
\prec a \ \mbox{and} \ (\ref{ap1}) \ \mbox{has nontrivial
solutions} \ \}
\end{equation}
The similar Theorem to Theorem \ref{pt1} is the following one.
\begin{theorem}\label{antipt1}
Let $n \in \natu$ and $a \in \tilde{\Lambda}_{n}$ be given and $u$
any nontrivial solution of (\ref{ap1}) such that $u(0) = u(T) =
0.$ If the zeros of $u$ in $[0,T]$ are denoted by $0 = x_0 < x_2 <
\ldots < x_{2m} = T$ and the zeros of $u'$ in $(0,T)$ are denoted
by $ x_1 < x_3 < \ldots < x_{2m-1}, $ then:
\begin{enumerate}
\item $x_{i+1} - x_{i} \leq \frac{T}{2(2n-1)}, \ \forall \ i: \ 0
\leq i \leq 2m-1.$ Moreover, at least one of these inequalities is
strict. \item $m$ is an odd number and $m \geq 2n+1.$ Any odd
value $m\geq 2n+1$ is possible. %
\item
\begin{equation}\label{a02071}
\Vert a - \tilde{\lambda}_{2n-1} \Vert_{L^1 (x_i,x_{i+1})} \geq
\frac{(2n-1)\pi}{T} \cot (\frac{(2n-1)\pi}{T}(x_{i+1}-x_i)),\ 0
\leq i \leq 2m-1.
\end{equation}
\item
\begin{equation}\label{ap1009071}
\tilde{\beta}_{1,n} \equiv \inf_{a \in \tilde{\Lambda}_{n} } \
\Vert a - \tilde{\lambda}_{2n-1} \Vert_{L^1 (0,T)} =
\frac{2\pi(2n-1)(2n+1)}{T} \cot \frac{(2n-1)\pi}{2(2n+1)}
\end{equation}
and $\tilde{\beta}_{1,n}$ is not attained.%
\item If $a \in \periodica$ satisfies
\begin{equation}\label{apositivo}
\tilde{\lambda}_{2n-1} \prec a, \ \Vert a \Vert_{L^1(0,T)} \leq
\tilde{\gamma}_{1,n} = T\tilde{\lambda}_{2n-1} +
\tilde{\beta}_{1,n},
\end{equation}
then
\begin{equation}\label{a1604096}
\tilde{\lambda}_{2n}(a) < 0 < \tilde{\lambda}_{2n+1}(a)
\end{equation}
\end{enumerate}
\end{theorem}

\begin{remark}\label{1307092} A similar theorem to the previous one may be proved if $a
\in \tilde{\Lambda}_0$ where
\begin{equation}
{\small \tilde{\Lambda}_0 = \{ a \in \periodica:  \ (\ref{ap1}) \
\mbox{has nontrivial solutions} \ \} }
\end{equation}
In this case $m\geq 1$ and any even value $m \geq 1$ is possible.
Consequently
\begin{equation}\label{1307093}
\tilde{\beta}_{1,0} = \inf_{a \in \tilde{\Lambda}_0} \ \Vert a^+
\Vert_{L^1 (0,T)} = \displaystyle \frac{4}{T}
\end{equation}
Let us remark that the restriction $ 0 \prec a $ which is natural
for the periodic problem (\ref{p1}), is not necessary in this case
(see Remark 4 in \cite{camovimia}).
\end{remark}

\section{Some applications}

\subsection{Stability of linear periodic
equations}\label{subsection1}

We begin with an application to the Lyapunov stability of the
Hill's equation
\begin{equation}\label{nuevahille}
u''(x) + a(x) u(x) = 0, \ a \in \periodica.
\end{equation}
To this respect, it is convenient to introduce the parametric
equation
\begin{equation}\label{2hilleq}
u''(x) + (\mu+a(x)) u(x) = 0, \ a \in \periodica, \ \mu \in \real.
\end{equation}
Remember that if $\lambda_i (a), \ i \in \natu \cup \{ 0 \}$ and
$\tilde{\lambda}_i (a), \ i \in \natu,$ denote, respectively the
eigenvalues of (\ref{nuevahille}) for the periodic and
antiperiodic problem, then it is known (\cite{cole}, \cite{hale})
that
\begin{equation}\label{2004097}
\lambda_0 (a) < \tilde{\lambda}_1 (a) \leq \tilde{\lambda}_2 (a) <
\lambda_{1}(a) \leq \lambda_{2}(a) < \tilde{\lambda}_{3}(a) \leq
\tilde{\lambda}_{4}(a) < \lambda_3 (a) \leq \ldots
\end{equation}
and that equation (\ref{2hilleq}) is stable if
\begin{equation}\label{2004095}
\mu \in (\lambda_{2n}(a),\tilde{\lambda}_{2n+1}(a)) \cup
(\tilde{\lambda}_{2n+2}(a),\lambda_{2n+1}(a))
\end{equation}
for some $n \in \natu \cup \{ 0 \}$ and that equation
(\ref{2hilleq}) is unstable if
\begin{equation}\label{2004096}
\mu \in (-\infty,\lambda_0 (a)] \cup
(\lambda_{2n+1}(a),\lambda_{2n+2}(a)) \cup
(\tilde{\lambda}_{2n+1}(a),\tilde{\lambda}_{2n+2}(a))
\end{equation}
for some $n \in \natu \cup \{ 0 \}.$ If $\mu = \lambda_{2n+1}(a)$
or $\mu = \lambda_{2n+2}(a)$, (\ref{2hilleq}) is stable if and
only if $\lambda_{2n+1}(a) = \lambda_{2n+2}(a)$ and, finally, if
$\mu = \tilde{\lambda}_{2n+1}(a)$ or $\mu =
\tilde{\lambda}_{2n+2}(a)$, (\ref{2hilleq}) is stable if and only
if $\tilde{\lambda}_{2n+1}(a) = \tilde{\lambda}_{2n+2}(a).$

\begin{theorem}\label{t0605091}
Let $a \in \periodica$  satisfying
\begin{equation}\label{borg}
\begin{array}{c}
 \exists \ p \in \mathbb{N}, \ \exists \ k \ \in [\frac{p^2 \pi^2}{T^2}, \frac{(p+1)^2 \pi^2}{T^2}]: \\ \\
 k \leq  a, \ \
 \Vert a \Vert_{L^1(0,T)}  \leq kT + k^{1/2}2(p+1)\cot \frac{k^{1/2}T}{2(p+1)}
 \end{array}
\end{equation}
Then $\mu = 0$ is in the $n^{th}$ stability zone of the Hill's
equation (\ref{2hilleq}).
\end{theorem}

\begin{proof} If $\frac{p^2\pi^2}{T^2} \equiv  a$ or $\frac{(p+1)^2\pi^2}{T^2} \equiv
a$ or $k \equiv a,$  then (\ref{2hilleq}) is trivially stable.
Therefore, we can assume
\begin{equation}\label{bisborg}
\begin{array}{c}
 \exists \ p \in \mathbb{N}, \ \exists \ k \ \in (\frac{p^2 \pi^2}{T^2}, \frac{(p+1)^2 \pi^2}{T^2}): \\ \\
 k \prec  a, \ \
 \Vert a \Vert_{L^1(0,T)}  \leq kT + k^{1/2}2(p+1)\cot \frac{k^{1/2}T}{2(p+1)}
 \end{array}
\end{equation}
In this case, the proof is a combination of different ideas used
in the previous two sections. Let us suppose, for instance, that
$p = 2n, \ n \in \mathbb{N}.$ As in Theorem \ref{pt1},
$\lambda_{2n}(a) < \lambda_{2n}(\lambda_{2n-1}) = 0.$ On the other
hand, since $k \prec a,$ doing a similar reasoning to that in
Theorem \ref{pt1}, but for the antiperiodic problem, we have that
if $u$ is a nontrivial solution of (\ref{ap1}) such that $u(0) =
u(T) =0,$ then $\vert x_{i+1}-x_i \vert \leq
\frac{\pi}{2k^{1/2}}.$ This implies the relation $m
> \frac{Tk^{1/2}}{\pi}$ in \ref{1307091}. But since we are now considering the antiperiodic problem
(\ref{ap1}), $m$ must be an odd number. Also $p
<\frac{Tk^{1/2}}{\pi} < p+1,$ and as $p =2n,$ we deduce $m \geq
2n+1.$ Consequently,
\begin{equation}\label{1105091}
\begin{array}{c}
\Vert a - k \Vert_{L^1 (0,T)} \geq k^{1/2} \displaystyle
\sum_{i=0}^{2m-1} \cot (k^{1/2}(x_{i+1}-x_i)) \geq \\ \\
k^{1/2}2m\cot (\frac{k^{1/2}T}{2m}) \geq k^{1/2}2(2n+1)\cot
(\frac{k^{1/2}T}{2(2n+1)})
\end{array}
\end{equation}
(see (\ref{2805091})).  Moreover, this last constant is not
attained. As in Theorem \ref{pt1}, if  $h: [0,1] \rightarrow
\mathbb{R}$ is defined as por $h(\varepsilon) =
\tilde{\lambda}_{2n+1}\left ( k + \varepsilon (a(\cdot) - k)
\right ),$ we obtain $h(0) > 0$ and $h(\varepsilon) \neq 0, \
\forall \ \varepsilon \in (0,1].$ Then, $h(1) =
\tilde{\lambda}_{2n+1}(a) > 0.$ As a consequence, $\mu = 0 \in
(\lambda_{2n}(a),\tilde{\lambda}_{2n+1}(a))$ and the Theorem is
proved. The proof is similar if $p$ is an odd number.
\end{proof}

\begin{remark} The case where $a(x) = \alpha + \beta \psi (x),$
with $\psi \in \periodica, \ \int_0^T \psi (x) \ dx = 0$ and
$\int_0 ^T \vert \psi (x) \vert \ dx = 1/T,$ was studied by Borg
(\cite{borg}). Borg used the characteristic multipliers determined
from Floquet's theory. He deduced stability criteria for
(\ref{2hilleq}) by using the two parameters $\alpha$ and $\beta.$
For a concrete function $a,$ this implies the use of the two
quantities
$$
\frac{1}{T} \int_0^T a(x) \ dx, \ \ \frac{1}{T} \left \Vert
a(\cdot) - \frac{1}{T} \int_0^T a(x) \ dx \right \Vert_{L^1 (0,T)}
$$
It is clear that the results given in Theorem \ref{t0605091} are
of a different nature (see \cite{magnus} and the translator's note
in \cite{krein}). In fact, our results are similar to those
obtained by Krein \cite{krein} by using a different procedure.
However, Krein assumed $k = \frac{p^2 \pi^2}{T^2}$ and an strict
inequality for $\Vert a \Vert_{L^1(0,T)}$ in (\ref{borg}) (see
Theorem 9 in \cite{krein}). By using Theorem \ref{antipt1} we can
assume a non strict inequality in (\ref{borg}) since the constant
$\tilde{\beta}_{1,n}$ is not attained.

Finally, if for a given function $a \in \periodica$ we know that
$a$ satisfies (\ref{bisborg}), the result given in Theorem
\ref{t0605091} is more precise than Krein's result since the
function
$$
kT + k^{1/2}2(p+1)\cot \frac{k^{1/2}T}{2(p+1)}, \ k \in \
[\frac{p^2 \pi^2}{T^2}, \frac{(p+1)^2 \pi^2}{T^2}]
$$
is strictly increasing.
\end{remark}

\begin{remark} The result obtained in previous Theorem uses $L_1$ Lyapunov
inequalities. In a similar way, if one uses $L_\infty$ Lyapunov
inequalities, the following result may be proved ( see
\cite{magnus}, Chapter V, Theorem 5.5). Here we take $T= \pi$ for
simplicity:

If $r$ and $s$ are given real numbers and
\begin{equation}\label{2905091}
r^2 \leq  a(x) \leq s^2
\end{equation}
then (\ref{nuevahille}) is stable for all possible functions $
a(\cdot)$ satisfying (\ref{2905091}) if and only if the interval
$(r^2,s^2)$ does not contain the square of an integer.

In particular, concerning to the first stability zone,
(\ref{nuevahille}) is stable if
\begin{equation}\label{2905092}
0 \leq a(x) \leq  1
\end{equation}
and for functions satisfying $0 \leq a(x),$ this result is optimal
in the following sense: for any positive number $\varepsilon$
there is some function $ a(x)$ with $a \in \periodica,$ satisfying
$0 \leq a(x) \leq 1+\varepsilon$ and such that (\ref{nuevahille})
is unstable.

We can exploit the results obtained in Theorem \ref{pt1} to obtain
new results on the stability properties of (\ref{nuevahille}).
This is the purpose of the next Theorem where function $a$ can be
uniformly greater than $1$ in an appropriate interval $(0,x_0)$ as
long as the length of the interval $(0,x_0)$ is sufficiently
small.
\end{remark}

\begin{theorem}\label{2905093}
Let us choose $T = \pi.$ If function $a$ fulfills

\begin{equation}\label{2905090}
\begin{array}{c}
a \in L_{\pi}(\real,\real), \  \ 0\prec a, \ \exists \alpha \in
(0,\frac{\pi}{2}), \ \exists \ x_0 \in \left (
\frac{\pi}{2}(1-\cos \alpha), \frac{\pi}{2}(1+\cos \alpha) \right
)
: \\ \\
\max \{ x_0^2 \Vert a \Vert_{L^\infty (0,x_0)},  \ (\pi-x_0)^2
\Vert a \Vert _{L^\infty(x_0,\pi)} \} \leq \alpha^2,
\end{array}
\end{equation}
 then (\ref{nuevahille}) is stable.
 \end{theorem}

\begin{proof}
The proof is based on the following two lemmas. The first one is
trivial but necessary. In the second Lemma we exploit the same
idea as in the continuation method used in the proof of the last
part of Theorem \ref{pt1}.

\begin{lemma}\label{2905094}
Let $\alpha \in (0, \frac{\pi}{2})$ and $ x_0 \in (0, \pi)$ be
given. If $a_{\alpha,x_0}$ is defined by
\begin{equation}\label{2005095}
a_{\alpha,x_0}(x) = \left \{
\begin{array}{l}
\frac{\alpha^2}{x_0^2},\ \mbox{if} \ x \in (0,x_0), \\
\frac{\alpha^2}{(\pi -x_0)^2},\ \mbox{if} \ x \in (x_0,\pi), \\
\end{array}
\right.
\end{equation}
then the antiperiodic boundary value problem
\begin{equation}\label{2905096}
u''(x) + a_{\alpha,x_0} (x) u(x) = 0, \ x \in (0,\pi), \ \ u(0) +
u(\pi) = u'(0) + u'(\pi) = 0,
\end{equation}
has nontrivial solutions if and only if $x_0 \in \{
\frac{\pi}{2}(1-\cos \alpha),\frac{\pi}{2}(1+\cos \alpha) \}.$
\end{lemma}
\begin{proof}
Taking into account the formula (\ref{2005095}) for the function
$a_{\alpha,x_0},$ any solution of (\ref{2905096}) must be of the
form
\begin{equation}\label{0106091}
u(x) = \left \{
\begin{array}{l}
A \sin \frac{\alpha x}{x_0} + B \cos \frac{\alpha x}{x_0}, \ 0
\leq x < x_0, \\ \\C \sin \frac{\alpha (x-\pi)}{\pi - x_0} + D
\cos \frac{\alpha (x-\pi)}{\pi - x_0},  x_0 < x \leq \pi
\end{array}
\right.
\end{equation}
Moreover, since $a_{\alpha,x_0} \in L^\infty (0,\pi),$ any
solution of (\ref{2905096}) is, in fact, a $C^1[0,\pi]$ function.
Therefore, it must satisfies the four conditions:
$$
\begin{array}{c}
u(0) + u(\pi) = 0, \ u'(0) + u'(\pi) = 0, \\ \\ \lim_{x
\rightarrow x_0^-} \ u(x) = \lim_{x \rightarrow x_0^+} \ u(x), \
\lim_{x \rightarrow x_0^-} \ u'(x) = \lim_{x \rightarrow x_0^+} \
u'(x).
\end{array}
$$
These four conditions are equivalent to the system of equations
$$
\begin{array}{c}
B + D = 0, \\ \\
\frac{\alpha}{x_0} A + \frac{\alpha}{\pi - x_0} C = 0, \\ \\ A
\sin \alpha + B \cos \alpha + C \sin \alpha - D \cos \alpha = 0,
\\ \\
\frac{A\alpha}{x_0} \cos \alpha - \frac{B\alpha}{x_0} \sin \alpha
- \frac{C\alpha}{\pi -x_0} \cos \alpha - \frac{D\alpha}{\pi -x_0}
\sin \alpha = 0
\end{array}
$$
The determinant of the previous system is equal to
$$
\frac{-4x_0^2 + 4x_0 \pi - \pi^2 \sin^2 \alpha}{x_0 ^2 (\pi -
x_0)^2} \ \alpha^2
$$
and it is zero if and only if $x_0 \in \{ \frac{\pi}{2}(1-\cos
\alpha),\frac{\pi}{2}(1+\cos \alpha) \}.$ The Lemma is proved.
\end{proof}

\begin{lemma}\label{2905097}
Let $\alpha \in (0, \frac{\pi}{2})$ be given and
\begin{equation}\label{0909091}
\begin{array}{c}
a \in L_{\pi}(\real,\real) \ \mbox{ such that} \ \exists \ x_0 \in
\left (
\frac{\pi}{2}(1-\cos \alpha),\frac{\pi}{2}(1+\cos \alpha) \right ): \\ \\
\max \{ x_0^2 \Vert a^+ \Vert_{L^\infty (0,x_0)},  \ (\pi-x_0)^2
\Vert a^+ \Vert _{L^\infty(x_0,\pi)} \} \leq \alpha^2
\end{array}
\end{equation}
then $\tilde{\lambda}_1 (a) > 0.$
\end{lemma}
\begin{proof}
Let $a_{\alpha,x_0}$ be the function defined in (\ref{2005095}).
If we define the continuous function $g: [0,1] \rightarrow \real,$
as $g(\varepsilon) = \tilde{\lambda}_1 (\varepsilon^2
a_{\alpha,x_0}),$ then $g(0) = \tilde{\lambda}_1 (0)=1.$ Moreover,
for each $\varepsilon \in (0,1],$ $\varepsilon^2 a_{\alpha,x_0} =
a_{\varepsilon \alpha,x_0}$ and since $\varepsilon \in (0,1],$ we
deduce from (\ref{0909091}) that $x_0 \notin \
\{\frac{\pi}{2}(1-\cos (\varepsilon \alpha),\frac{\pi}{2}(1+\cos
(\varepsilon \alpha) \}.$ From the previous Lemma we conclude that
the unique solution of the antiperiodic problem
\begin{equation}\label{0106093}
u''(x) + a_{\varepsilon \alpha,x_0} (x) u(x) = 0, \ x \in (0,\pi),
\ \ u(0) + u(\pi) = u'(0) + u'(\pi) = 0
\end{equation}
is the trivial one. Thus, $g(\varepsilon) \neq 0, \ \forall \
\varepsilon \in (0,1].$ Consequently $g(1) = \tilde{\lambda}_1
(a_{\alpha,x_0}) > 0.$ Since $a(x) \leq a_{\alpha,x_0}(x), \ x \in
(0,\pi),$  we have $\tilde{\lambda}_{1}(a) \geq \tilde{\lambda}_1
(a_{\alpha,x_0}).$
\end{proof}

$ $

By using these two Lemmas, the proof of the Theorem is trivial
since in the hypothesis (\ref{2905090}) is included the condition
$0 \prec a.$ This allows to prove that $ \mu = 0 \in
(\lambda_{0}(a), \tilde{\lambda}_1 (a)),$ the first stability zone
of (\ref{2hilleq}).
\end{proof}

\begin{remark}
It is trivially deduced from the previous proof that the
conclusion of Theorem \ref{2905093} is true if we assume the
hypothesis
\begin{equation}
\begin{array}{c}
a \in L_{\pi}(\real,\real) \setminus \{ 0 \} , \  \ \int_0^T a
\geq 0, \ \exists \alpha \in (0,\frac{\pi}{2}), \ \exists \ x_0
\in \left ( \frac{\pi}{2}(1-\cos \alpha), \frac{\pi}{2}(1+\cos
\alpha) \right )
: \\ \\
\max \{ x_0^2 \Vert a^+ \Vert_{L^\infty (0,x_0)},  \ (\pi-x_0)^2
\Vert a^+ \Vert _{L^\infty(x_0,\pi)} \} \leq \alpha^2,
\end{array}
\end{equation}
which is more general than (\ref{2905090}).
\end{remark}
\begin{remark}\label{0106094}
Taking into account the inequality
$$
\frac{\pi}{2}(1-\cos \alpha) < \alpha < \frac{\pi}{2}(1+\cos
\alpha),\ \forall \ \alpha \in (0,\frac{\pi}{2}),
$$
if we select in Theorem \ref{2905093}, $x_0 \in \left (
\frac{\pi}{2}(1-\cos \alpha),\alpha \right ),$ then the quantity
$\Vert a(\cdot) \Vert_{L^\infty (0,\pi)}$ can be greater than one.
Consequently, the $L_ \infty$ criterion  (\ref{2905092}) for the
stability of (\ref{nuevahille}) can not be applied. On the other
hand, for general $\alpha \in (0,\frac{\pi}{2}),$ we may choose
$x_0 = x_0 (\alpha)$ so that the quantity $\frac{\alpha}{x_0} $ is
as close as we want to $ \frac{2\alpha}{\pi (1-\cos \alpha)}.$
Since $\displaystyle \lim _{\alpha \rightarrow 0^+} \
\frac{2\alpha}{\pi (1-\cos \alpha)} = + \infty,$ the conclusion is
that the norm $\Vert a \Vert_{L^\infty (0,x_0 (\alpha))} $ may be
arbitrary large as long as the interval $(0,x_0 (\alpha))$ is
sufficiently small.

Finally, in Theorem \ref{2905093} the norm $\Vert a(\cdot)
\Vert_{L^1 (0,\pi)}$ may be chosen as near as we want to
$\frac{\alpha^2}{x_0} + \frac{\alpha^2}{\pi - x_0} =
\frac{\alpha^2 \pi}{x_0 (\pi -x_0)}$ and  if $\alpha \rightarrow
\frac{\pi}{2}^-,$ we have $\Vert a(\cdot) \Vert_{L^1 (0,\pi)}
\rightarrow \pi > \frac{4}{\pi}$. Consequently, the classical
Lyapunov's criterion for the stability of (\ref{nuevahille}) can
not be applied.
\end{remark}

\subsection{Nonlinear resonant problems}

We finish this paper with some new results on the existence and
uniqueness of solutions of nonlinear periodic b.v.p.
\begin{equation}\label{m38}
u''(x) + f(x,u(x)) = 0, \ x \in (0,T), \ \ u(0)- u(T)= u'(0)-
u'(T) = 0.
\end{equation}

Taking into account previous discussion (see Remark \ref{010609t1}
and Remark \ref{1212073} above), next theorem includes different
situations which can not be studied from the results in
\cite{wanglina95}, Theorem 6 and in \cite{huaizhongyongjde94},
Theorem 6. The proof, which uses similar ideas to that given in
\cite{camovimia} for the case of Neumann boundary conditions at
the first two eigenvalues, combines the linear results of the
previous sections with Schauder's fixed point theorem. We omit the
details.

\begin{theorem}\label{t0605092}
Let us consider (\ref{m38}) where:
\begin{enumerate}
\item $f$ and $f_{u}$ are Caratheodory function on $\real \times
\real$ and $f(x+T,u) = f(x,u), \ \forall \ (x,u) \in \real \times
\real.$%
\item There exist functions $\alpha, \ \beta \in L^\infty (0,T)$
satisfying
\begin{equation}\label{1205091}\lambda_{2n-1} \prec \alpha (x) \leq f_{u}(x,u) \leq
\beta (x), \ \Vert \beta \Vert_{L^1(0,T)} \leq \gamma_{1,n}
\end{equation}
\end{enumerate}
where $\gamma_{1,n}$ has been defined in (\ref{0909092}). Then,
problem (\ref{m38}) has a unique solution.
\end{theorem}

\begin{remark}
By using an example in (\cite{lazerleach}), it may be seen that
the restriction
$$
\lambda_{2n-1} \prec \alpha (x) \leq f_{u}(x,u) $$ in
(\ref{1205091}) cannot be replaced (in nonlinear problems) by the
weaker condition
$$\lambda_{2n-1} <  f_{u}(x,u) \leq \beta (x) $$
\end{remark}

$ $

The previous Theorem uses $L^1$ Lyapunov inequality. The last part
of this section is dedicated to show how we can again use the idea
of the continuation method used in Theorem \ref{pt1} and in Lemma
\ref{2905097}, to obtain new results on the existence and
uniqueness of solutions for resonant problems like (\ref{m38}), by
using $L^\infty$ Lyapunov inequalities. In this sense, it is very
well known (see \cite{huaizhongyongjde94}) that if, in addition to
the first hypothesis of the previous Theorem, $f_u$ fulfils the
nonresonance condition
\begin{equation}\label{010609t2}
\exists \ n \in \natu \cup \{ 0 \}, \lambda, \mu \in \real: \
\frac{(2n)^2 \pi^2}{T^2} < \lambda \leq f_u (x,u) \leq \mu <
\frac{(2(n+1))^2 \pi^2}{T^2},
\end{equation}
then (\ref{m38}) has a unique solution. In particular, if $n =0,$
(\ref{010609t2}) becomes
\begin{equation}\label{010609t3}
\exists \ \lambda, \mu \in \real: \ 0 < \lambda \leq f_u (x,u)
\leq \mu < \frac{4 \pi^2}{T^2}.
\end{equation}

To obtain an strict generalization of these results, we return to
the linear periodic problem (\ref{p1}). Again, we take $T =\pi$
for simplicity.

\begin{theorem}\label{p2905097}
If function $a$ satisfies
\begin{equation}
\begin{array}{c}
a \in L_{\pi}(\real,\real),  0 \prec a(x), \\ \\  \exists \ x_0
\in (0,\pi):  \max \{ x_0^2 \Vert a \Vert_{L^\infty (0,x_0)}, \
(\pi-x_0)^2 \Vert a \Vert _{L^\infty(x_0,\pi)} \} < \pi^2
\end{array}
\end{equation}
then $\lambda_0 (a) < 0 < \lambda_1 (a).$
\end{theorem}

\begin{proof}
The proof is based on the following Lemma, which points out an
important qualitative difference with respect to the antiperiodic
problem (see Lemma \ref{2905094}).

\begin{lemma}\label{p2905094}
Let $\alpha \in (0, \pi)$ and $ x_0 \in (0, \pi)$ be given. If
$a_{\alpha,x_0}$ is the function defined in (\ref{2005095}), the
periodic boundary value problem
\begin{equation}\label{p2905096}
u''(x) + a_{\alpha,x_0} (x) u(x) = 0, \ x \in (0,\pi), \ \ u(0) -
u(\pi) = u'(0) - u'(\pi) = 0
\end{equation}
has only the trivial solution.
\end{lemma}
\begin{proof}
As in the proof of the Lemma \ref{2905094}, taking into account
the formula (\ref{2005095}) for the function $a_{\alpha,x_0},$ any
solution of (\ref{p2905096}) must be of the form
\begin{equation}\label{p0106091}
u(x) = \left \{
\begin{array}{l}
A \sin \frac{\alpha x}{x_0} + B \cos \frac{\alpha x}{x_0}, \ 0
\leq x < x_0, \\ \\C \sin \frac{\alpha (x-\pi)}{\pi - x_0} + D
\cos \frac{\alpha (x-\pi)}{\pi - x_0},  \ x_0 < x \leq \pi
\end{array}
\right.
\end{equation}
Again, any solution of (\ref{p2905096}) is, in fact, a
$C^1[0,\pi]$ function. Therefore, it must satisfies the four
conditions:
$$
\begin{array}{c}
u(0) - u(\pi) = 0, \ u'(0) - u'(\pi) = 0, \\ \\ \lim_{x
\rightarrow x_0^-} \ u(x) = \lim_{x \rightarrow x_0^+} \ u(x), \
\lim_{x \rightarrow x_0^-} \ u'(x) = \lim_{x \rightarrow x_0^+} \
u'(x).
\end{array}
$$
These four conditions are equivalent to the system of equations
$$
\begin{array}{c}
B - D = 0, \\ \\
\frac{\alpha}{x_0} A - \frac{\alpha}{\pi - x_0} C = 0, \\ \\ A
\sin \alpha + B \cos \alpha + C \sin \alpha - D \cos \alpha = 0,
\\ \\
\frac{A\alpha}{x_0} \cos \alpha - \frac{B\alpha}{x_0} \sin \alpha
- \frac{C\alpha}{\pi -x_0} \cos \alpha - \frac{D\alpha}{\pi -x_0}
\sin \alpha = 0
\end{array}
$$
The determinant of the previous system is equal to
$$
\frac{\pi^2\alpha^2\sin^2 \alpha}{x_0 ^2 (\pi - x_0)^2}
$$
which is always different from zero.
\end{proof}
Now, to prove the Theorem \ref{p2905097}, let us define
$$
\alpha ^2 = \max \{ x_0^2 \Vert a \Vert_{L^\infty (0,x_0)},  \
(\pi-x_0)^2 \Vert a \Vert _{L^\infty(x_0,\pi)} \}
$$
Clearly $a(x) \leq a_{\alpha,x_0}(x), \ x \in (0,\pi),$ and
therefore $\lambda_{1}(a) \geq \lambda_1 (a_{\alpha,x_0}).$ On the
other hand, if we define the continuous function $g: [0,1]
\rightarrow \real,$ as $g(\varepsilon) = \lambda_1 (\varepsilon^2
a_{\alpha,x_0}),$ then $g(0) = \lambda_1 (0)
>0.$ Moreover, for each $\varepsilon \in (0,1],$ $\varepsilon^2
a_{\alpha,x_0} = a_{\varepsilon \alpha,x_0}$ and from the previous
Lemma, we deduce that the unique solution of the periodic problem
\begin{equation}\label{p0106093}
u''(x) + a_{\varepsilon \alpha,x_0} (x) u(x) = 0, \ x \in (0,\pi),
\ \ u(0) - u(\pi) = u'(0) - u'(\pi) = 0
\end{equation}
is the trivial one. Therefore $g(\varepsilon) \neq 0, \ \forall \
\varepsilon \in (0,1].$ Consequently $g(1) =\lambda_1
(a_{\alpha,x_0}) > 0$ and $\lambda_1 (a) >0.$
\end{proof}

\begin{remark}\label{p0106094}
If in the previous Theorem we select $x_0 \in (0,\pi/2),$ function
$a$ can satisfies $\Vert a \Vert_{L^\infty (0,x_0)} = \pi^2/x_0^2$
(which is a quantity greater than $ 4 $, see (\ref{010609t3}) for
$T=\pi$) as long as $\Vert a \Vert _{L^\infty (x_0,L)} <
\pi^2/(T-x_0)^2.$

By using the ideas of Theorem \ref{pt1} we can obtain analogous
results for the case of higher eigenvalues, i.e., in the case
where $\lambda_{2n-1} \prec a(x), \ n \in \natu.$
\end{remark}

The corresponding nonlinear Theorem to Theorem \ref{p2905097} is
the following one.

\begin{theorem}\label{tt0605092}
Let us consider (\ref{m38}) where:
\begin{enumerate}
\item $f$ and $f_{u}$ are Caratheodory function on $\real \times
\real$ and $f(x+T,u) = f(x,u), \ \forall \ (x,u) \in \real \times
\real.$%
\item There exist functions $\alpha, \ \beta \in L^\infty (0,T)$
satisfying
\begin{equation}\label{21205091} 0 \prec \alpha (x) \leq f_{u}(x,u) \leq
\beta (x).
\end{equation}
\item \begin{equation} \exists \ x_0 \in (0,\pi):  \max \{ x_0^2
\Vert \beta \Vert_{L^\infty (0,x_0)},  \ (\pi-x_0)^2 \Vert \beta
\Vert _{L^\infty(x_0,\pi)} \} < \pi^2
\end{equation}
\end{enumerate}
Then, problem (\ref{m38}) has a unique solution.
\end{theorem}

\end{document}